\documentclass{article}
\title{Coleman Maps for Modular Forms at Supersingular Primes over Lubin-Tate Extensions}
\author{Antonio Lei\footnote{Supported by Trinity College, Cambridge.}\\
\small Department of Pure Mathematics and\\
\small Mathematical Statistics, University of Cambridge\\
\small Wilberforce Road, Cambridge CB3 0WB, United Kingdom
}
\usepackage{amsmath}
\usepackage{amsthm}
\newtheorem{defn}{Definition}[section]
\newtheorem{thm}[defn]{Theorem}
\newtheorem{propn}[defn]{Proposition}
\newtheorem{cor}[defn]{Corollary}
\newtheorem{lem}[defn]{Lemma}
\newtheorem{rk}[defn]{Remark}
\newtheorem{conj}[defn]{Conjecture}
\newcommand{\vp}{\varphi}
\newcommand{\QQ}{\mathbb{Q}}
\newcommand{\ZZ}{\mathbb{Z}}
\newcommand{\Zp}{\mathbb{Z}_p}
\newcommand{\Gm}{\mathbb{G}_m}
\newcommand{\E}{\mathbb{E}_{h,V}}
\newcommand{\Qp}{\mathbb{Q}_p}

\newcommand{\DD}{\mathcal{D}}
\newcommand{\pp}{\mathfrak{p}}
\newcommand{\LL}{\mathcal{L}}
\newcommand{\HIw}{\mathbb{H}^1_{\mathrm{Iw}}}
\newcommand{\Hpm}{\mathbb{H}^1_{\mathrm{Iw},\pm}}
\newcommand{\col}{\mathrm{Col}^\pm}

\newcommand{\F}{\mathcal{F}}
\newcommand{\A}{\mathcal{A}}
\newcommand{\OF}{\mathcal{O}_F}
\newcommand{\OO}{\mathfrak{O}}
\newcommand{\z}{\mathbf{z}}
\newcommand{\eb}{\bar{\eta}}

\DeclareMathOperator{\Gal}{Gal}
\DeclareMathOperator{\Tw}{Tw}
\DeclareMathOperator{\Tr}{Tr}

\DeclareMathOperator{\Sel}{Sel}

\long\def\symbolfootnote[#1]#2{\begingroup
\def\thefootnote{\fnsymbol{footnote}}\footnote[#1]{#2}\endgroup}
\usepackage{latexsym}
\usepackage{amsfonts}
\usepackage[all]{xy}
\usepackage{amssymb}
\usepackage{amsmath}
\begin{document}
\maketitle
\begin{abstract}
Given an elliptic curve with supersingular reduction at an odd prime $p$, Iovita and Pollack have generalised results of Kobayashi to define even and odd Coleman maps at $p$ over Lubin-Tate extensions given by a formal group of height $1$. We generalise this construction to modular forms of higher weights. 
\end{abstract}
\symbolfootnote[0]{Email: aifl2@cam.ac.uk}
\symbolfootnote[0]{MSC 2000: 11R23; 11F11}
\symbolfootnote[0]{Keywords: Modular form; supersingular prime; Iwasawa theory; Lubin-Tate extension}
\section*{}
\setcounter{section}{-1}

%+++++++++++++++++++++++++++++++++++++++++++++++++++++++++++++++++++++++++
%+++++++++++++++++++++++++++++++++++++++++++++++++++++++++++++++++++++++++

\section{Introduction}
Let $f$ be a normalised eigen-newform of integral weight at least $2$ and $p$ an odd supersingular prime for $f$ (i.e. $p$ divides $a_p$ but not the level of $f$). On the one hand, the $p$-adic $L$-functions of $f$ defined in \cite{mtt} have unbounded coefficients. On the other hand, the $p$-Selmer group over the $\QQ_\infty$, the extension of $\QQ$ by adjoining all $p$ power roots of unity, is not $\Lambda$-cotorsion where $\Lambda$ is the Iwasawa algebra of $\Zp[[\Gal(\QQ_\infty/\QQ)]]$, which can be identified with the set of power series over $\Zp[{\rm Gal}(k_1/\Qp)]$. It makes the Iwasawa theory for $f$ at $p$ difficult. 

Much progress has been made in this direction. In \cite{po}, Pollack has defined the plus and minus analytic $p$-adic $L$-functions $L_p^\pm$, which have bounded coefficients in the case $a_p=0$. When $f$ corresponds to an elliptic curve $E$ defined over $\QQ$ and $p$ is as above, Kobayashi \cite{ko} defined the even and odd Selmer groups $\Sel_p^\pm(E/\QQ_\infty)$ by modifying the local condition of the usual Selmer group at $p$. These conditions are obtained by applying Honday theory to the formal group associated to $E$ at $p$. Kobayashi then used these conditions to construct
\begin{displaymath}
\col:\lim_{\leftarrow}H^1(k_n,T_E)\rightarrow\Lambda
\end{displaymath}
where $T_E$ is the Tate module of $E$ at $p$ and $k_n=\Qp(\mu_{p^n})$. It turns out that on applying $\col$ to the Kato zeta element defined in \cite{ka}, one obtains $L_p^\pm$, which can be used to show that $\Sel_p^\pm(E/\QQ_\infty)$ are $\Lambda$-cotorsion. It is then possible to formulate the ``main conjecture" in the following form:

\begin{conj}With the notation above, the characteristic ideal of the Pontryagin dual of $\Sel_p^\pm(E/\QQ_\infty)$ is generated by $L_p^\pm$.
\end{conj}

On the one hand, the construction of $\col$ was generalised by Iovita and Pollack \cite{ip} to Lubin-Tate extensions given by formal groups of height $1$. That is, we can replace $k_n$ by extensions of $\Qp$ obtained by adjoining torsion points of a Lubin-Tate group of height $1$ defined over $\Zp$. On the other hand, Kobayashi's construction can be generalised to modular forms of higher weights by using Perrin-Riou's exponential map (see \cite{l}). We will show that one can generalise the construction of the former to higher weight modular forms as well by using the Perrin-Riou's exponential map constructed by Zhang \cite{zh2}.

As in \cite{l}, instead of defining the Coleman maps using local conditions obtained from the formal group, we define the Coleman maps directly using the Perrin-Riou's exponential. We then obtain our new local conditions from $\ker(\col)$, which turn out to agree with the ones given by Kobayashi and Iovita-Pollack. We then use these conditions to define the corresponding Selmer groups.

We now outline the construction of $\col$ here. Let $V_f$ be the Deligne $p$-adic representation of $G_\QQ$ associated to $f$. Write $V=V_f(1)$, the Tate twist of $V_f$ and fix $T$ a lattice in $V$ which is stable under $G_\QQ$. Then, the Perrin-Riou's exponential map enables us to define two elements
\begin{displaymath}
\E(\mu_{\xi^\pm})\in\mathcal{H}_{(k-1)/2}\otimes\lim_{\leftarrow}H^1(k_n,T)
\end{displaymath}
where $\mathcal{H}_{(k-1)/2}$ denotes the set of power series over $\Qp[{\rm Gal}(k_1/\Qp)]$ which are of order $\log_p^{(k-1)/2}$. We then define 
\begin{eqnarray*}
\LL_{\xi^\pm}:\lim_{\leftarrow}H^1(k_n,T^*(1))&\rightarrow&\mathcal{H}_{(k-1)/2}\\
\z&\mapsto&<\E(\mu_{\xi^\pm}),\z>
\end{eqnarray*}
where $<,>$ is a pairing on
\begin{displaymath}
\Big(\mathcal{H}_{(k-1)/2}\underset{\Lambda}{\otimes}\lim_{\leftarrow}H^1(k_n,T)\Big)\times\lim_{\leftarrow}H^1(k_n,T^*(1))\rightarrow\mathcal{H}_{(k-1)/2}.
\end{displaymath}
On computing some of its special values, we show that $\LL_{\xi^\pm}(\z)$ is divisible by $\log_{p,k}^\pm$, which is defined in \cite{po} and has exact order $\log_p^{(k-1)/2}$. This enables us to define
\begin{eqnarray*}
\col:\lim_{\leftarrow}H^1(k_n,T^*(1))&\rightarrow&\QQ\otimes\Lambda\\
\z&\mapsto&\LL_{\xi^\pm}(\z)/\log_{p,k}^\pm.
\end{eqnarray*}

The structure of this paper is as follows. We will review results of \cite{zh2} in Section~\ref{zhang}. In particular, we will state the properties of the Perrin-Riou's exponential map which we will need for our construction of the Coleman maps. In Section~\ref{construct}, we will construct the Coleman maps using ideas from \cite{l}. The kernels and images of these maps will be described in Section~\ref{kernel} under certain technical assumptions. In particular, we will define the even and odd Selmer groups for some $\Zp$-extensions of a number field using our description of the kernels. Finally, we explain how the construction in Section~\ref{construct} can be generalised to relative Lubin-Tate groups in Section \ref{relative} using ideas of Kim (see \cite{ki}).

\textbf{Acknowledgements.} The author would like to thank Tony Scholl, Byoung Du Kim and Alex Bartel for the very helpful discussions. He is also indebted to the anonymous referees for their valuable suggestions.

%+++++++++++++++++++++++++++++++++++++++++++++++++++++++++++++++++++++++++
%+++++++++++++++++++++++++++++++++++++++++++++++++++++++++++++++++++++++++

\section{Perrin-Riou's exponential map over height 1 Lubin-Tate extensions}\label{zhang}
In \cite{zh2}, Zhang has generalised the construction of Perrin-Riou's exponential map defined in \cite{pr} to Lubin-Tate extensions. We review his results here.

We fix an odd prime $p$ and $\pi$ a uniformiser of $\Zp$. Let $\alpha$ be the $p$-adic unit in $\Zp^\times$ such that $\pi=\alpha p$. Let $g$ be a lift of Frobenius with respect to $\pi$, i.e. a power series over $\Zp$ such that $g(X)=\pi X+$(higher terms) and $g(X)\equiv X^p\mod{p}$. Then, $g$ gives rise to an one-dimensional height-one formal group over $\Zp$, which is independent of the choice of $g$ up to isomorphism over $\Zp$. We denote this formal group by $\F$.

We write $K=\Qp$ (reason being we want to replace $\Qp$ by a finite unramified extension of $\Qp$ in Section~\ref{relative}), $K_n$ denotes the extension of $K$ obtained by adjoining the $\pi^n$th roots of $\F$ and $G_n$ denotes the Galois group of $K_n$ over $K$ for $0\le n\le\infty$. In particular, $G_n\cong(\ZZ/p^n)^\times$ and $G_\infty\cong G_1\times\Gal(K_\infty/K_1)\cong\ZZ/(p-1)\times\Zp$.

Let $\kappa$ be the character of $G_K$ (the absolute Galois group of $K$) given by its action on the Tate module of $\F$. Then, $\sigma\omega=[\kappa(\sigma)]_\F(\omega)$ for all $\omega\in\F[\pi^\infty]$. If $\chi$ denotes the cyclotomic character of $G_K$, then $\kappa=\chi\psi$ for an unramified character $\psi$.

Let $\Xi$ denote the completion of the maximal unramified extension of $\Qp$ and $\OO$ its ring of integers. Let $\eta:\Gm\rightarrow\F$ be an isomorphism between the multiplicative group and $\F$. Then $\eta\in\OO[[X]]$. Moreover, $\eta(X)=\Omega X+$ (higher degree terms), where $\Omega$ is a $p$-adic unit. The lift of Frobenius $g$ satisfies $g\circ\eta=\eta^\vp\circ((1+X)^p-1)$ where $\vp$ is the Frobenius of $\Gal(\Qp^{\rm ur}/\Qp)$ which acts on $\eta$ by acting on its coefficients. In particular, $\Omega^\vp=\alpha\Omega$.

\begin{defn}
We define $\Xi[[X]]^\psi$ to be the set of power series $f$, defined over $\Xi$, such that $\sigma f(X)=f((1+X)^{\psi(\sigma)}-1)\forall\sigma\in G_K$.
\end{defn}
In particular, \cite[(1.13)]{zh2} says that $\eta\in\Xi[[X]]^\psi$. The significance of this set is given by the following:
\begin{lem}\label{eta}
Let $f\in\Xi[[X]]^\psi$ and $\zeta$ a $p^n$th root of unity. Then $f(\zeta-1)\in K_n$.
\end{lem}
\begin{proof}
By definition, $\sigma f(X)=f((1+X)^{\psi(\sigma)}-1)$ for any $\sigma\in G_K$. Therefore, we have
\begin{eqnarray*}
\sigma(f(\zeta-1))&=&(\sigma f)(\zeta^\sigma-1)\\
&=&f(\zeta^{\chi(\sigma)\psi(\sigma)}-1)\\
&=&f(\zeta^{\kappa(\sigma)}-1).
\end{eqnarray*}
If, in addition, $\sigma\in G_{K_n}$, then $\kappa(\sigma)\in1+p^n\Zp$. Hence, $\sigma(f(\zeta-1))=f(\zeta-1)$ for any $\sigma\in G_{K_n}$, so we are done.
\end{proof}
From now on, we fix a primitive $p^n$th root of unity $\zeta_{p^n}$ for each positive integer $n$ such that $\zeta_{p^{n+1}}^p=\zeta_{p^n}$. This determines an element $t\in B_{dR}^+$ (see \cite[Section~III.1]{co} for details). We also fix a crystalline (hence de Rham) representation $V$ of $G_K$ and write $D(V)=D_{dR}(V)=D_{\rm cris}(V)$ for its Dieudonn\'{e} module which is equipped with a de Rham filtration and an action of $\vp$. We denote the $i$th de Rham filtration by $D^i(V)$. We write $r(V)$ for the slope of $\vp$ on $D(V)$. Note that the action of $\vp$ extends to $\Xi\otimes D(V)$ naturally.

We write $V(k)$ for the $k$th Tate twist of $V$. Then, $D(V(k))=t^{-k}D(V)$ as $G_K$ acts on $t$ via $\chi$. Similarly, $D(V(\kappa^k))=t_\pi^{-k}D(V)$ where $t_\pi=\Omega t$ since $G_K$ acts on $t_\pi$ via $\kappa$ by \cite[Section 2]{zh2}. Their filtrations are given by the following:

\begin{lem}
The de Rham filtrations satisfy 
\[D^i(V(\kappa^j))=D^i(V(j))=t_\pi^{-j}D^{i+j}(V).\]
\end{lem}
\begin{proof}
Since $\Omega\in\bar{K}^\times$, we have
\begin{eqnarray*}
D^i(V(\kappa^j))&=&(t_\pi^{-j}D(V))\cap t^i B_{dR}^+\\
&=&t_\pi^{-j}(D(V)\cap t^{i+j}\Omega^j B_{dR}^+)\\
&=&t_\pi^{-j}(D(V)\cap t^{i+j} B_{dR}^+)\\
&=&t_\pi^{-j}D^{i+j}(V).
\end{eqnarray*}
Hence we are done.
\end{proof}

For $r\in\mathbb{R}_{\ge0}$, let $B$ be a Banach $p$-adic space, then $\DD_r(\Qp,B)$ denotes the set of tempered $B$-valued distributions of order $r$ (in the sense of \cite[Definition I.4.2]{co}) on the locally analytic functions with compact support in $\Qp$. It is equipped with a Galois action of $G_K$ as defined in \cite[(3.1)]{zh2}. Similarly, if $A$ is a compact open subset of $\Qp$, $\DD_r(A,B)$ denotes the set of tempered distributions of order $r$ on $A$ with values in $B$.

When $A=\Zp$, we write the Amice transform of $\mu\in\DD_r(\Zp,B)$ as $\A_\mu\in B[[X]]$, i.e.
\[
\A_\mu(X)=\int_{\Zp}(1+X)^x\mu(x).
\]

We define $\DD_r(\Qp,\Xi\otimes D(V))^\psi$ to be the subset of $\DD_r(\Qp,\Xi\otimes D(V))$ consisting of all the distributions $\mu$ satisfying:
\begin{displaymath}
\sigma\left(\int_{\Qp} f\mu\right)=\int_{\Qp} f(\psi(\sigma)x)\mu\ \forall\sigma\in G_K.
\end{displaymath}
\begin{rk}\label{amice}Let $\mu\in\DD_r(\Zp,\Xi\otimes D(V))$. Then, $\mu\in\DD_r(\Zp,\Xi\otimes D(V))^\psi$ iff its Amice transform is in $\Xi[[X]]^\psi\otimes D(V)$ (see \cite[Proposition 2.4(i)]{zh2}).\end{rk}

We define $\widetilde{\DD_r}(\Zp^\times,\Xi\otimes \DD(V))$ to be $\displaystyle\lim_{\underset{\Tw}{\leftarrow}}\DD_r\left(\Zp^\times,\Xi\otimes D(V(\kappa^k))\right)$ where $\Tw$ is the twist map given by $\mu\mapsto(-tx)^{-1}\mu$, which is well defined by \cite[Lemma~3.6]{zh1}. We define $\widetilde{\DD_r}(\Qp,\Xi\otimes \DD(V))$ similarly. In \cite[Theorems 3.3 and 3.6]{zh2}, the generalised Perrin-Riou's exponential is given by:

\begin{thm}\label{przh}
Let $h$ be a positive integer such that $D^{-h}(V)=D(V)$. Then, there is a map 
\begin{displaymath}\E:\widetilde{\DD_r}(\Qp,\Xi\otimes D(V))^{\vp_\DD\otimes\vp=1,\psi}\rightarrow H^1\left(K_\infty,\DD_{r+r(V)+h}(\Zp^\times,D(V))\right)^{G_\infty}\end{displaymath}
such that for $k\ge1-h$
\begin{eqnarray*}
\int_{\Zp^\times}x^k\E(\mu)&=&(k+h-1)!\exp_k\left((1-\vp)^{-1}(1-\frac{\vp^{-1}}{p})\int_{\Zp^\times}\frac{\mu}{(-tx)^k}\right),\\
\int_{1+p^n\Zp}x^k\E(\mu)&=&(k+h-1)!\exp_k\left(\frac{\vp^{-n}}{p^n}\int_{\Zp}\epsilon\left(\frac{x}{p^n}\right)\frac{\mu}{(-tx)^k}\right)
\end{eqnarray*}
where $\epsilon$ is as defined in \cite[Section V.1]{co} and $\exp_k$ denotes the exponential map for the $p$-adic representation $V(\kappa^k)$ as defined in \cite{bk}.
\end{thm}

%+++++++++++++++++++++++++++++++++++++++++++++++++++++++++++++++++++++++++
%+++++++++++++++++++++++++++++++++++++++++++++++++++++++++++++++++++++++++

\section{The construction of even and odd Coleman maps}\label{construct}
We construct $\col$ in three steps. First, we prove some elementary results about distributions on $\Zp^\times$ in Section~\ref{el}. In Section~\ref{sp}, we explain how to construct a measure $\mu_\xi\in\DD_0(\Zp^\times,\Xi\otimes D(V))^\psi$ from a given $\xi\in D(V)$ and compute some special values of $\E(\mu_\xi)$ using Theorem~\ref{przh} and results from Section~\ref{el} . Finally, in Section~\ref{mf}, we apply these results to a modular form $f$ by choosing two elements of $D(V_f)$, namely $\xi^\pm$. We then proceed as explained in the introduction to construct $\col$.

%+++++++++++++++++++++++++++++++++++++++++++++++++++++++++++++++++++++++++

\subsection{Distributions on $\Zp^\times$}\label{el}

Let $\mu\in\DD_r(\Zp,\Xi\otimes \DD(V))^{\psi}$, then $\mu\in\DD_r(\Zp^\times,\Xi\otimes \DD(V))^{\psi}$ iff
\[
\sum_{\zeta^p=1}\A_\mu(\zeta(1+X)-1)=0.
\]
On the space of power series satisfying this condition, $D=(1+X)\frac{d}{dX}$ acts bijectively. Moreover, for such $\mu$, we have 
\begin{equation}
D^k\A_\mu(\zeta_{p^n}-1)=\int_{\Zp^\times}\epsilon\left(\frac{x}{p^n}\right)x^k\mu\label{substitute},\\
\end{equation}
see e.g. \cite[Section~I.5]{co}.
\begin{lem}\label{lift}
Any $\mu\in\DD_r(\Zp^\times,\Xi\otimes \DD(V))^{\psi}$ can be lifted to 
\[
\widetilde{\mu}\in\widetilde{\DD_r}(\Qp,\Xi\otimes \DD(V))^{\vp_\DD\otimes\vp=1,\psi}.\] Moreover, the image of such a lift under $\E$ is independent of the choice of the lift.
\end{lem}
\begin{proof}
\cite[Lemma IX.2.8 and Remark IX.2.6(iii)]{co} and \cite[Lemma~3.5]{zh2}.
\end{proof}

Given any $\mu\in\DD_r(\Zp^\times,\Xi\otimes \DD(V))^{\psi}$, we abuse notation and write $\E(\mu)=\E(\widetilde{\mu})$ where $\widetilde{\mu}$ is a lift given by Lemma~\ref{lift}. The fact that $\vp_\DD\otimes\vp(\widetilde{\mu})=\widetilde{\mu}$ implies that
\begin{equation}\label{invariant}
\int_{pA}f(x)\widetilde{\mu}=\vp\left(\int_Af(px)\widetilde{\mu}\right)
\end{equation}
for any $f$ and $A\subset\Qp$. It allows us to compute some special values of $\widetilde{\mu}$.

\begin{lem}\label{constantterms}With the above notation,
$\int_{\Zp}x^k\widetilde{\mu}=(1-p^k\vp)^{-1}\left(D^k\A_\mu(0)\right)$.
\end{lem}
\begin{proof}Since $\widetilde{\mu_\xi}$ restricts to $\mu_\xi$ on $\Zp^\times$, \eqref{substitute} implies that
\[
\int_{\Zp^\times}x^k\widetilde{\mu_\xi}=\int_{\Zp^\times}x^k\mu_\xi=D^k\A_\mu(0).
\]
Hence, by applying \eqref{invariant} to the decomposition
\[
\int_{\Zp}x^k\widetilde{\mu}=\int_{p\Zp}x^k\widetilde{\mu}+\int_{\Zp^\times}x^k\widetilde{\mu},
\]
we have
\[
\int_{\Zp}x^k\widetilde{\mu}=p^k\vp\left(\int_{\Zp}x^k\widetilde{\mu}\right)+D^k\A_\mu(0).
\]
\end{proof}

\begin{lem}\label{specialvalues}With the notation above,
\[\int_{\Zp}\epsilon\left(\frac{x}{p^n}\right)x^k\widetilde{\mu}=\sum_{i=0}^{n-1}p^{ik}\vp^{i}\left(D^k\A_\mu(\zeta_{p^{n-i}}-1)\right)+p^{nk}(1-p^k\vp)^{-1}(D^k\A_\mu(0)).\]
\end{lem}

\begin{proof}
Since $\Zp=\Zp^\times\cup p\Zp^\times\cup\cdots\cup p^{n-1}\Zp^\times\cup p^n\Zp$, we have
\begin{eqnarray*}
&&\int_{\Zp}\epsilon\left(\frac{x}{p^n}\right)x^k\widetilde{\mu}\\
&=&\sum_{i=0}^{n-1}\int_{p^i\Zp^\times}\epsilon\left(\frac{x}{p^n}\right)x^k\mu+\int_{p^n\Zp}\epsilon\left(\frac{x}{p^n}\right)x^k\widetilde{\mu}\\
&=&\sum_{i=0}^{n-1}p^{ik}\vp^{i}\left(\int_{\Zp^\times}\epsilon\left(\frac{x}{p^{n-i}}\right)x^k\mu\right)+p^{nk}\vp^n\int_{\Zp}x^k\widetilde{\mu}
\end{eqnarray*}
where the last equality follows from repeated applications of \eqref{invariant}. Hence the result by \eqref{substitute} and Lemma~\ref{constantterms}.
\end{proof}

%+++++++++++++++++++++++++++++++++++++++++++++++++++++++++++++++++++++++++

\subsection{Computing some special values}\label{sp}

With the notation above, we define
\begin{displaymath}\eb(X)=\eta(X)-\frac{1}{p}\sum_{\zeta^p=1}\eta(\zeta(1+X)-1).\end{displaymath}
Then $\sum_{\zeta^p=1}\eb(\zeta(1+X)-1)=0$. Moreover, we have:

\begin{lem}\label{eb}
We have $\eb\in\Xi[[X]]^\psi$.
\end{lem}
\begin{proof}
Let $\sigma\in G_{\Qp}$ and $\zeta$ a $p$th root of unity. By \cite[(1.13)]{zh2}, $\eta\in\Xi[[X]]^\psi$ and $\sigma\eta(X)=\eta((1+X)^{\psi(\sigma)}-1)$. If we replace $X$ by $\zeta^\sigma(1+X)-1$, we have
\begin{eqnarray*}
\sigma(\eta(\zeta(1+X)-1))&=&(\sigma\eta)(\zeta^\sigma(1+X)-1)\\
&=&\eta(\left(\zeta^{\sigma}(1+X)\right)^{\psi(\sigma)}-1)\\
&=&\eta(\zeta^{\kappa(\sigma)}(1+X)^{\psi(\sigma)}-1)
\end{eqnarray*}
Hence, on summing over $\zeta^p=1$, we have
\begin{eqnarray*}
\sigma\left(\sum_{\zeta^p=1}\eta(\zeta(1+X)-1)\right)&=&\sum_{\zeta^p=1}\sigma(\eta(\zeta(1+X)-1))\\
&=&\sum_{\zeta^p=1}\eta(\zeta^{\kappa(\sigma)}(1+X)^{\psi(\sigma)}-1)\\
&=&\sum_{\zeta^p=1}\eta(\zeta(1+X)^{\psi(\sigma)}-1)\ ({\rm as\ }\kappa(\sigma)\in\Zp^\times).
\end{eqnarray*}
Hence, the sum $\sum_{\zeta^p=1}\eta(\zeta(1+X)-1)\in\Xi[[X]]^\psi$, so we are done.
\end{proof}

Let $\xi\in \DD(V)$, then $\eb(X)\otimes\xi$ defines an element $\mu_\xi\in\DD_0(\Zp^\times,\Xi\otimes\DD(V))$ with
\begin{displaymath}
\eb(X)\otimes\xi=\int_{\Zp^\times}(1+X)^x\mu_\xi.
\end{displaymath}
By Lemma \ref{eb} and Remark \ref{amice}, $\mu_\xi\in \DD_0(\Zp^\times,\Xi\otimes \DD(V))^\psi$. On applying the Perrin-Riou's exponential, we have:

\begin{propn}\label{special}
With the notation above, we have for $n\ge1$ and $k\ge 1-h$
\[
\int_{1+p^n\Zp}(-x)^k\E(\mu_\xi)=(k+h-1)!\exp_k\left(\gamma_{n,k}(\xi)\right)
\]
where $\gamma_{n,k}(\xi)$ is defined by
\[
\frac{1}{p^n}\left(\sum_{i=0}^{n-i}D^{-k}\eb^{\vp^{i-n}}(\zeta_{p^{n-i}}-1)\otimes\vp^{i-n}(\xi_k)+(1-\vp)^{-1}(D^{-k}\eb(0)\otimes\xi_k)\right)
\]
with $\xi_k=\xi t^{-k}$.
\end{propn}
\begin{proof}
The result follows from combining Theorem~\ref{przh} with Lemmas~\ref{constantterms} and \ref{specialvalues} and the fact that $\vp(t)=pt$.
\end{proof}

Our assumption on the eigenvalues of $\vp$ implies that there is an isomorphism
\begin{eqnarray*}
H^1(K_\infty,\DD_r(Z_p^\times,V))^{G_\infty}&\cong&\DD_r(G_\infty)\otimes\HIw(V)\\
\mu&\mapsto&\left(\int_{1+p^n\Zp}\mu\right)_n
\end{eqnarray*}
where $\HIw(V):=\displaystyle\lim_{\underset{\rm cor}{\leftarrow}}H^1(K_n,V)$ and $\DD_r(G_\infty)=\DD_r(G_\infty,\Qp)$ (see e.g. \cite[Proposition 2]{co}). Under this identification, we have
\[
\E(\mu_\xi)\in \DD_{h+r(V)}(G_\infty)\otimes\HIw(V).
\]

Write $\Tw_k:\HIw(V)\rightarrow\HIw(V(\kappa^k))$ for the twist map. Recall that $\Tw_k(\mu)=(-tx)^{-k}\mu$, so Proposition~\ref{special} implies that if $n\ge1$ and $k\ge1-h$, the $n$th component of $\Tw_k(\E(\mu))$ is given by
\begin{equation}
(k+h-1)!\exp_{k}(\gamma_{n,k}(\xi))
\end{equation}
where $\exp_{k}$ now denotes the exponential map $K_n\otimes\DD(V(\kappa^k))\rightarrow H^1(K_n,V(\kappa^k))$.

Recall that $G_\infty\cong G_1\times \Gamma$ where $\Gamma\cong\Zp$. We fix a topological generator $\gamma$ of $\Gamma$, then $\DD_r(G_\infty)$ can be identified with the set of power series in $\gamma-1$ over $\Qp[G_1]$ which are $O(\log_p^r)$.

We now assume that $V$ has a $F$-vector space structure where $F$ is a finite extension of $\Qp$ and the action of $G_K$ commutes with the multiplication by $F$. Denote the ring of integers of $F$ by $\OF$. Let $\Lambda=\OF[[G_\infty]]=\displaystyle\lim_\leftarrow\OF[G_n]$, then there is a pairing
\begin{eqnarray*}
<,>:\HIw(V)\times\HIw(V^*(1))&\rightarrow&\QQ\otimes\Lambda\\
\left((x_n)_n,(y_n)_n\right)&\mapsto&\left(\sum_{\sigma\in G_n}[x_n^\sigma,y_n]_n\sigma\right)_n
\end{eqnarray*}
where $[,]_n$ is the pairing on $H^1(K_n,V)\times H^1(K_n,V^*(1))\rightarrow F$. It extends to
\begin{displaymath}
\Big(\DD_{m}(G_\infty)\underset{\Lambda}{\otimes}\HIw(V)\Big)\times\Big(\DD_{n}(G_\infty)\underset{\Lambda}{\otimes}\HIw(V^*(1))\Big)\rightarrow\DD_{m+n}(G_\infty)
\end{displaymath}
for all $m,n\in\mathbb{R}_{\ge0}$. This enables us to define the following:
\begin{defn}
For a fixed $\xi\in D(V)$, we define a map
\begin{eqnarray*}\LL_\xi^h:\HIw(V^*(1))&\rightarrow&\DD_{r(V)+h}(G_\infty)\\
\z&\mapsto&<\E(\mu_\xi),\z>.\end{eqnarray*}
\end{defn}

Following the calculations of \cite{ku}, we find that for $n\ge1$, the $n$th component of $\Tw_k\LL_\xi(\z)$ is given by:
\begin{eqnarray*}
\left(\Tw_k\LL_\xi^h(\z)\right)_n&=&(h+k-1)!\sum_{\sigma\in G_n}[\exp_k(\gamma_{n,k}(\xi)^\sigma),z_{-k,n}]_n\sigma\\
&=&(h+k-1)![\sum_{\sigma\in G_n}\gamma_{n,k}(\xi)^\sigma\sigma,\sum_{\sigma\in G_n}\exp^*_k(z_{-k,n}^\sigma)\sigma^{-1}]_n
\end{eqnarray*}
where $z_{-k,n}$ denotes the image of $\z$ under
\begin{displaymath}
\HIw(V^*(1))\rightarrow\HIw(V^*(1)(\kappa^{-k}))\rightarrow H^1(K_n,V^*(1)(\kappa^{-k}))
\end{displaymath} and $\Tw_k$ acts on $\DD_{r(V)+h}(G_\infty)$ by $\sigma\mapsto\kappa(\sigma)^k\sigma$ for $\sigma\in G_\infty$.

Let $\theta$ be a character on $G_n$ which does not factor through $G_{n-1}$. Since $D^{-k}\eb^{\vp^{i-n}}(\zeta_{p^{n-i}}-1)\in K_{n-i}$ by Lemma~\ref{eta}, we have 
\begin{displaymath}
\theta\left(\sum_{\sigma\in G_n}\gamma_{n,k}(\xi)^\sigma\sigma\right)=\frac{1}{p^n}\sum_{\sigma\in G_n}D^{-k}\eb^{\vp^{-n}}(\zeta_{p^n}-1)^\sigma\theta(\sigma)\otimes\vp^{-n}(\xi_k).
\end{displaymath}
Hence, as in \cite[Lemma 1.4]{l}, we have
\begin{equation}\label{chara}
\begin{split}
&\frac{1}{(h+k-1)!}\kappa^k\theta(\LL_\xi^h(\z))\\=\ &\frac{1}{p^n}\left[\sum_{\sigma\in G_n}D^{-k}\eb^{\vp^{-n}}(\zeta_{p^n}-1)^\sigma\theta(\sigma)\otimes\vp^{-n}(\xi_k),\sum_{\sigma\in G_n}\exp^*_k(z_{-k,n}^\sigma)\theta(\sigma^{-1})\right]_n.\end{split}
\end{equation}
%+++++++++++++++++++++++++++++++++++++++++++++++++++++++++++++++++++++++++

\subsection{Modular forms}\label{mf}
From now on, we fix a normalised newform $f=\sum a_nq^n$ of integral weight $k\ge2$ with $p$ a supersingular prime for $f$ and $a_p=0$ (i.e. $p$ divides $a_p$ but not the level of $f$). We allow the character of $f$ to be arbitrary, but for the sole purpose of easing notation, we assume that the character of $f$ takes value $1$ at $p$. Let $V_f$ be the Deligne representation of $G_\QQ$ defined in \cite{de}. Let $L=\QQ(a_n:n\ge1)$ be the field of coefficients of $f$ and fix a place of $L$ above $p$. Then, $V$ is a two-dimensional vector space over $F=L_v$ and the action of $G_\QQ$ commutes with $F$. If we take $V$ to be $V_f(1)$, the Frobenius $\vp$ on $D(V)$ satisfies
\begin{displaymath}
\vp^2-\frac{a_p}{p}\vp+p^{k-3}=0.
\end{displaymath}
In particular, $r(V)=(k-1)/2-1$ and the assumption that the eigenvalues of $\vp$ on $D(V_f)$ are not integral powers of $p$ is automatically satisfied. On taking $h=1$ in Theorem~\ref{przh} and writing $\LL_\xi$ for $\LL_\xi^h$, we have Im$(\LL_\xi)\subset\DD_{(k-1)/2}(G_\infty)$ for any $\xi\in D(V)$.

The de Rham filtration of $D(V_f)$ is given by
\begin{displaymath}
D^i(V_f)=D^0(V_f(i))=\left\{
\begin{array}{ll}
D(V_f)&\mathrm{if\ \ }i\leq0\\
0&\mathrm{if\ \ }i\ge k\\
F\cdot \omega&\mathrm{if\ \ }1\le i\le k-1.
\end{array}\right.
\end{displaymath}
where $\omega$ is any non-zero element of $D^1(V_f)=D^0(V)$. We fix one such $\omega$, this corresponds to a choice of periods for $f$ (see \cite{ka}). We have $D^0(V(j))=D^0(V(\kappa^j))=F\cdot\omega$ for $0\le j\le k-2$.

Let $\gamma=\kappa(u)$, then we can define $\log_{p,k}^\pm$ as in \cite{po}:
\begin{eqnarray*}
\log_{p,k}^+&=&\prod_{j=0}^{k-2}\prod_{n=1}^\infty\frac{\Phi_{2n}(\gamma^{-j}u)}{p},\\
\log_{p,k}^-&=&\prod_{j=0}^{k-2}\prod_{n=1}^\infty\frac{\Phi_{2n-1}(\gamma^{-j}u)}{p},
\end{eqnarray*}
where $\Phi_m$ denotes the $p^m$th cyclotomic polynomial. In particular, the zeros of $\log_{p,k}^+$ are given by $\kappa^j\theta$ where $0\le j\le k-2$ and $\theta$ is a character of $G_n$ which does not factor through $G_{n-1}$ with $n$ odd, whereas those of $\log_{p,k}^-$ are characters of the same form but with even $n$. Moreover, $\log_{p,k}^\pm$ have exact order $\log_p^{\frac{k-1}{2}}$. We can now give a generalisation of \cite[Lemma 2.2]{l}:
\begin{lem}
Let $\xi^+=\vp(\omega)$ and $\xi^-=\omega$, then $\log_{p,k}^\pm|\LL_{\xi^\pm}(\z)$ for all $\z\in\HIw(V^*(1))$.
\end{lem}
\begin{proof}
We have $\vp^{2n}(\omega)\in D^0(V(\kappa^r))$ for all integers $n$ and $0\le r\le k-2$. Therefore, by (\ref{chara}), we have
\begin{eqnarray*}
\kappa^r\theta(\LL_{\xi^+}(\mathbf{z}))&=\ \ 0&\mathrm{\ \ if\ }n\mathrm{\ is\ odd},\\
\kappa^r\theta(\LL_{\xi^-}(\mathbf{z}))&=\ \ 0&\mathrm{\ \ if\ }n\mathrm{\ is\ even}
\end{eqnarray*}
where $\theta$ is a character of $G_n$ which does not factor through $G_{n-1}$. Hence, the zeros of $\log_{p,k}^\pm$ are also zeros of $\LL_{\xi^\pm}(\z)$, so we are done.
\end{proof}
In particular, since $\LL_{\xi^\pm}(\z)\in\DD_{(k-1)/2}(G_\infty)$, we have $\LL_{\xi^\pm}(\z)/\log_{p,k}^\pm=O(1)$. Hence, we have:
\begin{defn}
The even and odd Coleman maps are defined to be
\begin{eqnarray*}
\col:\HIw(V^*(1))&\rightarrow&\QQ\otimes\Lambda\\
\z&\mapsto&\LL_{\xi^\pm}(\z)/\log_{p,k}^\pm.
\end{eqnarray*}
\end{defn}

%+++++++++++++++++++++++++++++++++++++++++++++++++++++++++++++++++++++++++
%+++++++++++++++++++++++++++++++++++++++++++++++++++++++++++++++++++++++++

\section{Kernel}\label{kernel}
In this section, we describe the kernels of $\col$, generalising those given in \cite{ip} and use them to define the even and odd Selmer groups. We first give some elementary linear algebra results.

%+++++++++++++++++++++++++++++++++++++++++++++++++++++++++++++++++++++++++

\subsection{Linear algebra} 
For any positive integer $n$, we write $\pi_n=\eta^{\vp^{-n}}(\zeta_{p^n}-1)$. Then, $g^{(n)}(\pi_n)=0$ where $g^{(n)}=\underbrace{g\circ\cdots\circ g}_n$. Moreover, $g(\pi_n)=\pi_{n-1}$ and $K_n=K(\pi_n)$. We will from now on assume $g$ to be a good lift of Frobenius in the sense of \cite[Section 4.1]{ip}. In particular, we will have to assume $\pi\in p(1+p\Zp)$ which would exclude many Lubin-Tate extensions of $\Qp$. However, if we start with a totally ramified $\Zp$-extension of $\Qp$, then we can always assume that it is obtained from such Lubin-Tate extensions (see \cite{ip} for details). For $n>1$, let $\pi_{n}'=\pi_{n}-\frac{1}{p}\Tr_{n/n-1}(\pi_{n})=\pi_n+1$ and $\pi_1'=\pi_1-\frac{1}{p-1}\Tr_{1/0}(\pi_1)=\pi_1+\frac{p}{p-1}$. Then, $\Tr_{n/n-1}(\pi_n')=0$ for all $n\ge1$. 

\begin{lem}\label{normal}
Let $K^{(n)}$ be the kernel of the trace map from $K_n$ to $K_{n-1}$, then $\{\pi_n'^{\sigma}:\sigma\in G_n\}$ generates $K^{(n)}$ over $K$. 
\end{lem}
\begin{proof}
Let $x\in K^{(n)}$. By \cite[Proposition 4.4]{ip}, we have $x\in K[G_n]\pi_n+K_{n-1}$. Since $\Tr_{n/n-1}\pi_n\in K_{n-1}$, we can write $x=\sum_{\sigma\in G_n}a_\sigma\pi_n'^\sigma+y$ for some $a_\sigma\in K$ and $y\in K_{n-1}$. Since $\Tr_{n/n-1}x=\Tr_{n/n-1}\pi_n'^\sigma=0$ for all $\sigma$, we have $y=0$. Hence we are done.
\end{proof}

\begin{cor}\label{spanning}
Let $n\ge0$ be an integer and $\displaystyle\alpha=\sum_{i=0}^nx_i\pi_i'$ for some $x_i\in K$ with $\pi_0'=1$. Then, the $k$-vector space generated by $\{\alpha^\sigma:\sigma\in G_n\}$ is given by $\underset{i\in S}{\oplus}K^{(i)}$ where $S=\{i:x_i\ne0\}$ and $K^{(0)}=K$.
\end{cor}
\begin{proof}
We proceed by induction on $|S|$. The case $|S|=1$ follows directly from Lemma \ref{normal}.

Without loss of generality, we assume that $x_n\ne0$. Let $\displaystyle\beta=\sum_{i=0}^{n-1}x_i\pi_i'$. Then, by induction,  $\{\beta^\tau:\tau\in G_{n-1}\}$, generates $\underset{i\in S\setminus\{n\}}{\oplus}K^{(i)}$ over $K$. Fix $\tau\in G_{n-1}$ and consider the following $p$ elements: $\alpha^\sigma$, $\sigma|_{K_{n-1}}=\tau$. Then, their sum equals $p\beta^\tau+(\Tr_{n/n-1}\pi_n')^\tau=p\beta^\tau$. Therefore, for any $\tau\in G_{n-1}$ and $\sigma\in G_n$, $\beta^\tau$ and $\pi_n'^\sigma$ lie inside the $K$-vector space generated by $\alpha^\sigma$. Hence we are done.\end{proof}

%+++++++++++++++++++++++++++++++++++++++++++++++++++++++++++++++++++++++++
\subsection{Description of the kernels}

We now fix a lattice $T_f$ in $V_f$ which is stable under $G_K$. Write $T=T_f(1)\subset V=V_f(1)$. To describe the kernel of $\col$, we will assume $p\ge k-1$ as in \cite{l}. This implies that $(V/T(\kappa^m))^{G_{K_n}}=0$ for any $j$ and $n$ as in \cite[Lemma 2.5]{l}. Therefore, $H^1(K_n,T(\kappa^m))$ injects into $H^1(K_n,V(\kappa^m))$ under the natural map and we can treat the former as a lattice of the latter. In addition, the corestriction maps between $H^1(K_n,T(\kappa^m))$ are surjective and the restriction maps are injective (see \cite{ko}). We will treat $H^1(K_n,T(\kappa^m))$ as a subset of $H^1(K_{n'},T(\kappa^m))$ for $n'\ge n$.

Let $\z\in\HIw(T^*(1))$, then $\z\in\ker(\col)$ iff $z_{-m,n}$ is in the annihilator of the $\OF$-module generated by $\left\{\exp_m(\gamma_{n,m}(\xi^\pm)^\sigma):\sigma\in G_n\right\}$ for all $n\ge0$ and $0\le m\le k-2$. By \cite[Proposition 2.7]{l}, this is in fact equivalent to the same statement being true for all, $n\ge0$ with one fixed $m\in\{0,\ldots,k-2\}$ (we will take $m=0$ below).

Instead of looking at the said $\OF$-module, we study the $F$-vector space generated by these elements inside $H^1_f(K_n,V(\kappa^m))$ first. We can then intersect it with $H^1_f(K_n,T(\kappa^m))$ to obtain the kernel.

\begin{propn}
The vector subspace over $F$ of $H^1_f(K_n,V(\kappa))$ generated by the set $\left\{\exp(\gamma_{n,0}(\xi^\pm)^\sigma):\sigma\in G_n\right\}$, is equal to
\begin{displaymath}
\left\{x\in H^1_f(K_n,V):\mathrm{cor}_{n/m+1}x\in H^1_f(K_m,V)\forall m\mathrm{\ even\ (odd)}\right\}.
\end{displaymath} 
\end{propn}
\begin{proof}
Recall that by the proof of Lemma \ref{eta}, we have $\sigma f(\zeta-1)=f(\zeta^{\kappa(\sigma)}-1)$ for any $f\in\Xi[[X]]^\psi$, $\sigma\in G_K$ and $\zeta$ a $p$ power root of unity. Therefore, for $n>1$
\begin{displaymath}
\sum_{\zeta^p=1}f(\zeta\zeta_{p^n}-1)=\Tr_{n/n-1}f(\zeta_{p^n}-1).
\end{displaymath}
If $n=1$, then 
\begin{displaymath}
\sum_{\zeta^p=1}f(\zeta\zeta_{p}-1)=f(0)+\Tr_{1/0}f(\zeta_{p}-1).
\end{displaymath}
Hence, we have
\begin{eqnarray*}
p^n\gamma_{n,0}(\xi)&=&\sum_{i=0}^{n-1}\eb^{\vp^{i-n}}(\zeta_{p^{n-i}}-1)\otimes\vp^{i-n}(\xi)+\eb(0)\otimes(1-\vp)^{-1}(\xi)\\
&=&\sum_{i=0}^{n-1}\left(\eta^{\vp^{i-n}}(\zeta_{p^{n-i}}-1)-\frac{1}{p}\sum_{\zeta^p=1}\eta^{\vp^{i-n}}(\zeta\zeta_{p^{n-i}}-1)\right)\otimes\vp^{i-n}(\xi)\\
&&+\left(\eta(0)-\frac{1}{p}\sum_{\zeta^p=1}\eta(\zeta-1)\right)\otimes(1-\vp)^{-1}(\xi)\\
&=&\sum_{i=0}^n\left(\pi_{n-i}-\frac{1}{p}\Tr(\pi_{n-i})\right)\otimes\vp^{i-n}(\xi)-\frac{1}{p}\Tr(\pi_{1})\otimes(1-\vp)^{-1}(\xi)\\
&=&\sum_{i=0}^n\pi_{n-i}'\otimes\vp^{i-n}(\xi)-\frac{1}{p-1}\otimes\xi+(1-\vp)^{-1}(\xi).
\end{eqnarray*}
Recall that $\vp^2=-p^{k-3}$, so we have
\begin{displaymath}
(1-\vp)^{-1}=\frac{1}{1+p^{k-3}}(1+\vp).
\end{displaymath}
In particular, $-\frac{1}{p-1}\otimes\xi^\pm+(1-\vp)^{-1}(\xi^\pm)\notin D^0(V)$. Moreover, $\vp^r(\omega)\in D^0(V)$ iff $r$ is even, hence $\{\gamma_{n,0}(\xi^\pm)^\sigma\}$ generates
\begin{displaymath}
\left(K+\sum_{i\in S^\pm}K^{(i)}\right)\otimes D(V)/D^0(V)
\end{displaymath}
where $S^\pm=\{m\in[1,n]:m\ \rm even\ (odd)\}$ by Corollary \ref{spanning}. Hence the result by \cite[Lemma 2.8]{l}.
\end{proof}
We write $H^1_f(K_n,V)^\pm$ for the vector space described in the proposition and define $H^1_f(K_n,T)^\pm=H^1_f(K_n,T)\cap H_f^1(K_n,V)^\pm$. Then,
\begin{displaymath}
H^1_f(K_n,T)^\pm=\left\{x\in H^1_f(K_n,T):\mathrm{cor}_{n/m+1}x\in H^1_f(K_m,T)\forall m\mathrm{\ even\ (odd)}\right\}
\end{displaymath}
and $\ker(\col)$ is given by  
\begin{displaymath}
\Hpm(T^*(1)):=\lim_{\leftarrow}H^1_\pm(K_n,T^*(1))
\end{displaymath}
where $H^1_\pm(K_n,T^*(1))$ is defined to be the annihilator of $H^1_f(K_n,T)^\pm$ under the pairing
\begin{displaymath}
H^1(K_n,T^*(1))\times H^1(K_n,T)\rightarrow\OF.
\end{displaymath}

The images of $\col$ can be found in the same way as \cite[Section~3]{l}. Namely, $\displaystyle{\rm Im(Col^+)}\cong(u-1)\Lambda+\sum_{\sigma\in G_1}\Lambda$ and ${\rm Im(Col^-)}\cong\Lambda$.

%+++++++++++++++++++++++++++++++++++++++++++++++++++++++++++++++++++++++++

\subsection{The even and odd Selmer groups}

Let $E$ be a number field  with $[E:\QQ]=d$. Then, the $p$-Selmer group of $f$ over $E$ is defined to be
\[
\Sel_p(f/E)=\ker\left(H^1(E,V/T)\rightarrow\prod_v\frac{H^1(E_v,V/T)}{H^1_f(E_v,V/T)}\right)
\]
where $v$ runs through all places of $E$ and $V$ and $T$ are as defined above.

Assume that $p$ splits completely in $E$. Let $\pp_1,\ldots,\pp_d$ be the primes of $E$ above $p$ and $E_\infty/E$ a $\Zp$-extension such that $\pp_i$ is totally ramified in $E_\infty$. We write $E_n$ for the $n$th layer. Note that $E_{\pp_i}$ is isomorphic to $\Qp$ for $i=1,\ldots,d$. By \cite[Section~4.2]{ip}, $E_{\infty,\pp_i}/E_{\pp_i}$ is contained in a Lubin-Tate extension for some uniformiser $\pi$ of $\Qp$ such that $\pi\in p(1+p\Zp)$. Therefore, the $\col$ restrict to $\displaystyle\lim_{\leftarrow}H^1(E_{n,\pp_i},T^*(1))$ and it easy to check that the description of the kernels generalise directly. For each $n\ge0$, we can define
\[
\Sel_p^\pm(f/E_n)=\ker\left(\Sel_p(f/E)\rightarrow\prod_{i}\frac{H^1(E_{n,\pp_i},V/T)}{H^1(E_{n,\pp_i},T)^\pm\otimes\Qp/\Zp}\right)
\]
and $\Sel_p^\pm(f/E_\infty)=\displaystyle\lim_{\rightarrow}\Sel_p^\pm(f/E_n)$.

Unfortunately, unlike the cyclotomic case, $\Sel_p^\pm(f/E_\infty)$ is not $\Lambda$-cotorsion in general. However, they do satisfy a control theorem (cf \cite[Theorem~9.3]{ko}) and their coranks can be used to describe those of $\Sel_p(f/E_n)$ (cf \cite[Proposition~7.1]{ip}). Since the proofs for these results given in \cite{ip,ko} are purely algebraic and do not involve properties of elliptic curves, they generalise to general $f$ with no difficulties.

%+++++++++++++++++++++++++++++++++++++++++++++++++++++++++++++++++++++++++
%+++++++++++++++++++++++++++++++++++++++++++++++++++++++++++++++++++++++++

\section{Relative Lubin-Tate groups}\label{relative}
We now assume $K$ to be a finite unramified extension of $\Qp$ of degree $d$. For a fixed $\pi\in\Zp$ with $p$-adic valuation $d$, let $g$ be a lift of Frobenius with respect to $\pi$ in the sense of \cite[Section~I.1.2]{des}, then $\vp^i(g)$ is also such a lift for any integer $i$. To ease notation, we will write $g_i$ for $\vp^i(g)$. Each $g_i$ gives rise to an one-dimensional formal group over $\mathcal{O}_K$ which we write as $\F_{g_i}$. For any positive integer $n$, we write
\begin{displaymath}
g^{(n)}_i=\vp^{n-1}(g_i)\circ\vp^{n-2}(g_i)\circ\cdots\circ g_i=g_{i+n-1}\circ g_{i+n-2}\circ\cdots\circ g_i.
\end{displaymath}
Let $W^n_{g_i}$ be the set of zeros of $g^{(n)}_i$ in $\bar{K}$ and write $K_n=K(W^n_{g_i})$ which is independent of the choice of $g$ and $i$. Moreover, if $\omega\in W^n_{g_i}\setminus W^{n-1}_{g_i}$, then $K_n=K(\omega)$. Let $\eta_i:\Gm\rightarrow F_{g_i}$ be an isomorphism, then $\eta_i\in\OO[[X]]$ and $\omega_{n,i}:=\eta_i^{\vp^{-n}}(\zeta_{p^n}-1)\in W^n_{g_{i-n}}\setminus W^{n-1}_{g_{i-n}}$ (see \cite[Section~I.3.2]{des}). Note that $g_{i-n}$ sends $W^n_{g_{i-n}}$ to $W^{n-1}_{g_{i-n+1}}$, we define the Tate module of $F_{g_i}$ to be
\begin{displaymath}
T_{g_i}=\lim_{\underset{g_{i-n}}{\longleftarrow}}W^n_{g_{i-n}}.
\end{displaymath}
Since $\eta_i$ satisfies $g_i\circ\eta_i=\eta_i^\vp((1+X)^p-1)$, we have $(\omega_{n,i})_n\in T_{g_i}$.

The character $\kappa$ of $G_K$ on $T_{g_i}$ is independent of $i$ by \cite[Proposition I.1.8]{des}. As in the case of absolute Lubin-Tate groups, $\kappa$ can be decomposed as $\kappa=\chi\psi$ where $\chi$ is the cyclotomic character and $\psi$ is an unramified character.

Results of \cite{zh2} hold in this context with the obvious modifications, especially Theorem~\ref{przh}. In particular, for any $\xi\in D(V)$ and $i$ an integer, we can define a measure $\mu^{(i)}_\xi$ on $\Zp^\times$ whose Amice transform is given by $\bar{\eta_i}(X)\otimes\xi$ where $\bar{\eta_i}$ is defined in the same way as $\eb$ in Section~\ref{construct}. We can then define $\LL_\xi^{(i)}$ as before. For $V=V_f(1)$ and $F=\Qp$ (so $\OF=\Zp$), we define
\begin{eqnarray*}
\col:\HIw(V^*(1))&\rightarrow&\QQ\otimes\Lambda^d\\
\z&\mapsto&\left(\LL_{\xi^\pm}^{(i)}(\z)/\log_{p,k}^\pm\right)_{i=0,\cdots,d-1}.
\end{eqnarray*}

We now follow \cite[Section~3]{ki} to find the image of $\mathrm{Col}^-$. In particular, we assume that $g$ is a polynomial of degree $p$ and the coefficient of $X^{p-1}$ is $\zeta_0 p$ where $\zeta_0$ is a root of unity in $K$ such that $\mathcal{O}_K=\Zp[\zeta_0]$.
\begin{lem}\label{ind}
With the above notation, $\left(\E(\mu_{\xi^-}^{(i)})\right)_0$, $i=0,\cdots,d-1$, is linearly independent over $\Qp$.
\end{lem}
\begin{proof}
By Theorem \ref{special}, we have
\begin{displaymath}
\left(\E(\mu_{\xi^-}^{(i)})\right)_0=\exp\left((1-\vp)^{-1}\left(1-\frac{\vp^{-1}}{p}\right)\bar{\eta_i}(0)\otimes\xi^-\right).
\end{displaymath}
We first simplify the expression $(1-\vp)^{-1}(1-\frac{\vp^{-1}}{p})$. Recall that $\vp$ satisfies
\begin{displaymath}
\vp^2+p^{k-3}=0\ \ {\rm and}\ \ (1-\vp)^{-1}=\frac{1}{1+p^{k-3}}(1+\vp).
\end{displaymath}
Therefore,
\begin{eqnarray*}
&&(1-\vp)^{-1}\left(1-\frac{\vp^{-1}}{p}\right)\\
&=&\frac{1}{1+p^{k-3}}(\vp+1)\left(1-\frac{\vp^{-1}}{p}\right)\\
&=&\frac{1}{1+p^{k-3}}\left(\vp-\frac{\vp^{-1}}{p}+1-\frac{1}{p}\right)\\
&=&\frac{1}{1+p^{k-3}}\left(\left(1+\frac{1}{p^{k-2}}\right)\vp+1-\frac{1}{p}\right).
\end{eqnarray*}
We write $\lambda=(p^{2-k}+1)/(p^{k-3}+1)$. Since $\xi^-=\omega\in D^0(V)$, we have
\begin{displaymath}
(1-\vp)^{-1}\left(1-\frac{\vp^{-1}}{p}\right)\bar{\eta_i}(0)\otimes\xi^-\equiv\lambda\bar{\eta_i}^\vp(0)\otimes\vp(\omega)\mod{D^0(V)}.
\end{displaymath}
But $\bar{\eta_i}^\vp(0)$ equals to
\begin{displaymath}
\eta_i^\vp(0)-\frac{1}{p}\sum_{\zeta^p=1}\eta_i^\vp(\zeta-1)=\vp^{i+1}(\zeta_0)
\end{displaymath}
since the summands are the roots $g_i^\vp$. By definition, $\zeta_0,\vp(\zeta_0)\cdots,\vp^{d-1}(\zeta_0)$ is a $\Zp$-basis of $\mathcal{O}_K$, so we are done. 
\end{proof}
\begin{cor}
The image of $\HIw(T^*(1))$ under $\mathrm{Col}^-$ is isomorphic to $\Lambda^d$.
\end{cor}
\begin{proof}
By \cite[proof of Lemma 3.11]{l}, there exists an integer $r$ such that 
\begin{displaymath}
p^{-r}\left(\E(\mu_{\xi^-}^{(i)})\right)_0\in H^1(K,T)\setminus pH^1(K,T)
\end{displaymath}
for all $i$. Hence, as in \cite[proof of Proposition 3.9]{ki}, their linear independence over $\Zp$ implies that
\begin{displaymath}
\left\{\left(p^{-r}\LL_{\xi^-}^{(i)}(z)\right)_{i=0,\cdots,d-1}:z\in H^1(K,T)\right\}=\Zp^d.
\end{displaymath}
But the image of $\log_{p,k}^-$ in $\Zp$ is a $p$-adic unit (see \cite[Section 3.2]{l}), so we have
\begin{displaymath}
p^{-r}{\rm Col}^-_0(H^1(K,T^*(1)))=\Zp^d.
\end{displaymath}
But the following diagram commutes (see \cite[proof of Theorem 3.10]{l}):
\begin{displaymath}
\xymatrix{
H^1(K_m,T^*(1))\ar[d]^{\rm cor}\ar[rr]^{\ \ \ p^{-r}\LL_{\xi^-,m}^{(i)}}&&\Qp[G_m]\ar[d]^{\rm pr}\ar[rr]^{(\log_{p,k}^-)^{-1}}&&\Zp[G_m]\ar[d]^{\rm pr}\\
H^1(K_n,T^*(1))\ar[rr]^{\ \ \ p^{-r}\LL_{\xi^-,n}^{(i)}}&&\Qp[G_n]\ar[rr]^{(\log_{p,k}^-)^{-1}}&&\Zp[G_n]
}
\end{displaymath}
where $m>n$, hence the result by Nakayama's lemma.
\end{proof}

%+++++++++++++++++++++++++++++++++++++++++++++++++++++++++++++++++++++++++
%+++++++++++++++++++++++++++++++++++++++++++++++++++++++++++++++++++++++++

\bibliography{references}

\begin{thebibliography}{99}
\bibitem{bk}Bloch, S., Kato, K.:
\emph{$L$-functions and Tamagawa numbers of motives}, The Grothendieck Festschrift, Vol. I,  333--400, Progr. Math., 86, Birkhäuser Boston, Boston, MA, 1990.
\bibitem{co}Colmez, P.:
\emph{Th\'{e}orie d'Iwasawa des représentations de de Rham d'un corps local}, Ann. of Math. (2)  148  (1998),  no. 2, 485--571.
\bibitem{de}Deligne, P.:
\emph{Formes modulaires et repr\'{e}sentations $l$-adiques}, S\'{e}minaire Bourbaki, 11 (1968-1969), Exp. No. 355, 139--172.
\bibitem{des}de Shalit, E.:
\emph{Iwasawa theory of elliptic curves with complex multiplication}, Perspectives in Mathematics, 3. Academic Press, Inc., Boston, MA, 1987.
\bibitem{ip}Iovita, A., Pollack, R.:
\emph{Iwasawa theory of elliptic curves at supersingular primes over $\Bbb Z\sb p$-extensions of number fields},  J. Reine Angew. Math.  598  (2006), 71--103.
\bibitem{ka}Kato, K.:
\emph{$p$-adic Hodge theory and values of zeta functions of modular forms}, Cohomologies $p$-adiques et applications arithm\'{e}tiques. III, Ast\'{e}risque No. 295 (2004), ix, 117--290.
\bibitem{ki}Kim, B.D.:
\emph{The parity conjecture for elliptic curves at supersingular reduction primes}, Compos. Math.  143  (2007),  no. 1, 47--72.
\bibitem{ko}Kobayashi, S.:
\emph{Iwasawa theory for elliptic curves at supersingular primes}, Invent. Math.  152  (2003),  no. 1, 1--36.
\bibitem{ku}Kurihara, M.:
\emph{On the Tate Shafarevich groups over cyclotomic fields of an elliptic curve with supersingular reduction. I.}, Invent. Math.  149  (2002),  no. 1, 195--224. 
\bibitem{l}Lei, A.:
\emph{Iwasawa theory for modular forms at supersingular primes}, arXiv:0904.3938v2 [math.NT]
\bibitem{mtt}Mazur, B., Tate, J., Teitelbaum, J.:
\emph{On $p$-adic analogues of the conjectures of Birch and Swinnerton-Dyer}, Invent. Math.  84  (1986),  no. 1, 1--48.
\bibitem{pr}Perrin-Riou, B.:
\emph{Th\'{e}orie d'Iwasawa des repr\'{e}sentations $p$-adiques sur un corps local},  Invent. Math.  115  (1994),  no. 1, 81--161.
\bibitem{po}Pollack, R.:
\emph{On the $p$-adic $L$-function of a modular form at a supersingular prime}, Duke Math. J.  118  (2003),  no. 3, 523--558.
\bibitem{zh1}Zhang, S.:
\emph{On a trivial zero problem}, Int. J. Math. Math. Sci.  2004,  no. 5-8, 295--318. 
\bibitem{zh2}Zhang, S.:
\emph{On explicit reciprocity law over formal groups},  Int. J. Math. Math. Sci.  (2004),  no. 9-12, 607--635.
\end{thebibliography}

\end{document}